\documentclass[11pt, a4paper]{amsart}
\usepackage{amsmath, amssymb, amsthm}
\usepackage{color}
\usepackage{tikz}
\usepackage{tikz-cd}
\usepackage{mathrsfs}

\calclayout

\newif\ifdebug

\debugtrue

\newcommand{\SK}{\mathcal{K}}
\newcommand{\momang}{\mathcal{Z}_{\mathcal{K}}}
\newcommand{\cmomang}{\mathcal{Z}_{\Sigma}}
\newcommand{\FoM}{\varOmega(\momang)}
\newcommand{\FoMin}{\varOmega(\momang)^{T^m}}
\newcommand{\NB}{\mathcal{N}}
\newcommand{\CM}{\mathcal{C}}
\newcommand{\W}{\mathcal{W}}

\newcommand{\Hom}{\mathrm{Hom}}
\def\Ker{\mathop{\mathrm{Ker}}}
\def\Tor{\mathop{\mathrm{Tor}}\nolimits}
\def\hb{H_{\mathrm{bas}}}

\def\textgoth{\mathfrak}
\def\C{\mathbb C}
\def\R{\mathbb R}

\def\g{\textgoth g}
\def\h{\textgoth h}

\newcommand{\mb}[1]{{\textbf {\textit#1}}}

\def\ge{\geqslant}

\def\le{\leqslant}
\def\leq{\leqslant}

\theoremstyle{plain}
\newtheorem{thm}{Theorem}[section]
\newtheorem{lemma}[thm]{Lemma}

\newtheorem{proposition}[thm]{Proposition}

\theoremstyle{definition}
\newtheorem{defn}[thm]{Definition} 

\newtheorem{rmk}[thm]{Remark}
\newtheorem{constr}[thm]{Construction}

\begin{document}
\title[Basic cohomology of canonical holomorphic foliations]{Basic cohomology of canonical holomorphic foliations on complex moment-angle manifolds}

\author{Hiroaki Ishida}
\address{Department of Mathematics and Computer Science, Graduate School of Science
and Engineering, Kagoshima University, 890-0065 Kagoshima, Japan}
\email{ishida@sci.kagoshima-u.ac.jp}

\author{Roman Krutowski}
\address{Faculty of Mathematics, International Laboratory of Mirror Symmetry
and Automorphic Forms, National Research University Higher School of Economics, 
119048 Moscow, Russia}
\email{roman.krutovskiy@protonmail.com}

\author{Taras Panov}
\address{Department of Mathematics and Mechanics, Lomonosov Moscow
State University, Leninskie gory, 119991 Moscow, Russia;\newline
Institute for Information Transmission Problems, Russian Academy of Sciences, Moscow;\newline
Institute of Theoretical and Experimental Physics, Moscow}
\email{tpanov@mech.math.msu.su}
\thanks{This work was supported by the Grant-in-Aid for Young Scientists (B) (16K17596), and joint research projects and joint seminars under the bilateral program ``Topology and geometry of torus actions, cohomological rigidity, and hyperbolic manifolds" from Japan Society for the Promotion of Science (to H.\,I.); the Laboratory of Mirror Symmetry National Research University Higher School of Economics (RF Government grant, ag. No. 14.641.31.0001), the Simons Foundation, the Moebius Contest Foundation for Young Scientists, and the Foundation for the Advancement of Theoretical Physics and Mathematics ``BASIS'' (to R.\,K.); the Russian Foundation for Basic Research (grants no. 20-01-00675, 18-51-50005), and a Simons IUM Fellowship (to T.\,P.).
The authors thank the organisers of the conference ``Glances at Manifolds~2018'' in Krakow for giving us the opportunity to meet and exchange ideas that laid the foundation for this work.}


\begin{abstract}
We describe the basic cohomology ring of the canonical holomorphic foliation on a moment-angle manifold, LVMB-manifold or any complex manifold with a maximal holomorphic torus action. Namely, we show that the basic cohomology has a description similar to the cohomology ring of a complete simplicial toric variety due to Danilov and Jurkiewicz. This settles a question of Battaglia and Zaffran, who previously computed the basic Betti numbers for the canonical holomorphic foliation in the case of a shellable fan. Our proof uses an Eilenberg--Moore spectral sequence argument; the key ingredient is the formality of the Cartan model for the torus action on a moment-angle manifold. We develop the concept of transverse equivalence as an important tool for studying smooth and holomorphic foliated manifolds.
For an arbitrary complex manifold with a maximal torus action, we show that it is transverse equivalent to a moment-angle manifold and therefore has the same basic cohomology.
\end{abstract}

\maketitle

\section{Introduction}
The \emph{moment-angle complex} $\momang$ corresponding to a simplicial complex~$\SK$ is a topological space formed by gluing together products of discs and circles according to a recipe determined by the combinatorics of~$\SK$, see \cite{buchstaber2000torus,buchstaber2015toric}. The space $\momang$ carries a natural torus action, and $\momang$ is a manifold when $\SK$ is a triangulated sphere. An even-dimensional moment-angle manifold $\cmomang$ admits a complex structure invariant under the torus action if and only if $\SK$ is the underlying complex of a complete simplicial fan~$\Sigma$ (in other words, $\SK$ is a star-shaped sphere triangulation)~\cite{panov2012complex,tamb12,ishida2013complex}. 

Complex moment-angle manifolds are a subclass of \emph{LVMB-manifolds}. The latter is a class of  non-K\"ahler complex manifolds with holomorphic torus action introduced by Bosio~\cite{bosio2001varietes} as  a generalisation of \emph{LVM-manifolds} defined via intersections of Hermitian quadrics in the works of L\'opez de Medrano--Verjovsky~\cite{lo-ve97} and Meerssemann~\cite{meer00}. It was proved in~\cite[Theorem~9.4]{ishida2013complex} that complex moment-angle manifolds are biholomorphic to LVMB-manifolds of certain type, namely to those LVMB-manifolds for which $0$ is indispensable in the corresponding LVMB-datum. The LVMB-datum can be thought of as the collection of cones Gale dual to the cones in the simplicial fan defining a complex moment-angle manifold. In this formalism, LVM-manifolds correspond to simplicial fans which are  normal fans of convex polytopes.

LVMB-manifolds (and therefore complex moment-angle manifolds) are examples of \emph{complex manifolds with maximal torus action}, completely classified in~\cite{ishida2013complex} in terms of simplicial fans. If a compact torus $T$ acts on a smooth manifold $M$ effectively, then the following inequality holds for any $x\in M$: 
\[
  \dim T+\dim T_x\le\dim M,
\]
where $T_x$ denotes the stabiliser at~$x$. A $T$-action is called \emph{maximal} if the above turns into an equality for some $x\in M$. Examples of compact complex manifolds with a maximal action of torus by holomorphic transformations include holomorphic tori, Hopf and Calabi--Eckmann manifolds, moment-angle manifolds, LVM- and LVMB-manifolds, and nonsingular complete toric varieties (\emph{toric manifolds}). Moment-angle manifolds (or LVMB-manifolds with $0$ indispensable) are universal in the classification of~\cite{ishida2013complex}, in the sense that any complex manifolds with a maximal torus action is the quotient of a complex moment-angle manifold by a free action of closed subgroup of torus.

Battaglia and Zaffran~\cite{battaglia2015foliations} considered a certain holomorphic foliation on a complex LVMB-manifold, which later was shown in to be a particular case of the \emph{canonical} foliation on any complex manifold with a holomorphic torus action~\cite{ishida2017torus}. We review the construction of this canonical foliation in Section~\ref{prelim}. Battaglia and Zaffran computed the basic Betti numbers for their foliation in the case when the associated complete fan is shellable. Their method consisted in applying the Mayer--Vietoris sequence. They conjectured that the basic cohomology ring has a description similar to the cohomology ring of a complete simplicial toric variety due to Danilov and Jurkiewicz~\cite{danilov1978geometry}. The conjecture was justified by the fact that in the case of a complete regular fan the foliation becomes a locally trivial bundle over the associated toric variety with fibre a holomorphic torus (see Remark~\ref{holofib}).

In this paper we first prove the conjecture for all complex moment-angle manifolds with invariant complex structure (see Theorem~\ref{basicTh}). Our approach is different from that of Battaglia--Zaffran: we use the Eilenberg--Moore spectral sequence and establish the formality of the Cartan model for the torus action on $\momang$ (see Lemma~\ref{formality}). 

In the second part of the paper we study the notion of transverse equivalence for foliated manifolds. It is useful, as the basic cohomology rings of transverse equivalent foliated manifolds are isomorphic. We adapt the notion of transverse equivalence to our situation of complex manifolds with maximal torus actions and their canonical foliations (see Definition~\ref{defn:equivalent}). We use the classification results of the first author~\cite{ishida2013complex} for complex manifolds with maximal torus actions to show that the transverse equivalence class of such a manifold is determined by its marked fan data (Theorem~\ref{thm:fundamental}). As a consequence, we prove that any complex manifold with a maximal torus action is transverse equivalent to a complex moment-angle manifold (Theorem~\ref{thm:ess.surjective}). This gives a description of the basic cohomology ring for any complex manifold with a maximal torus action (Theorem~\ref{basiccohgen}). Since LVMB manifolds are a particular class of maximal torus actions, the conjecture of Battaglia and Zaffran is proved completely. 

An important consequence is that the basic cohomology ring of a complex manifold with a maximal torus action depends only on the combinatorics of the underlying simplicial fan, and does not depend on the other pieces of data defining the complex structure. This gives additional justification for considering the transverse equivalence classes of maximal torus actions as proper analogues of complete toric varieties in the non-rational case~\cite{k-l-m-v}.

We note that our approach can be also applied to smooth moment-angle manifolds 
$\momang$ rather than complex ones. Nevertheless, it is important to emphasise the holomorphic nature of the foliation under consideration. The reason is that we hope that our methods can be applied for calculation of the basic Dolbeault cohomology for the foliation and Dolbeault cohomology of complex moment-angle manifolds. Recently these rings were computed  for the case when the foliation is transverse K{\"a}hler, that is, when the fan $\Sigma$ is polytopal. 
In the general case, the description of the Dolbeault cohomology rings is an open problem, which we shall address in a subsequent work.

We thank the anonymous referees for their valuable comments and useful suggestions on improving the text.

\section{Preliminaries}\label{prelim}

\subsection{The moment-angle complex}
An abstract \emph{simplicial complex} on the set $[m]=\{1,2,\ldots,m\}$ is a collection \( \mathcal{K} \) of subsets $ I \subset [m]$ such that if $I \in \mathcal{K}$ then each $J \subset I$ also belongs to \( \mathcal{K} \).
We assume that the empty set $\varnothing$ is in~$\mathcal K$. A one-element subset 
$\{i\}\subset[m]$ is a \emph{vertex} if $\{i\}\in\mathcal K$; otherwise it is a \emph{ghost vertex}.

The \emph{moment-angle complex} $\mathcal{Z}_{\mathcal{K}}$ corresponding to $\mathcal{K}$ is a topological space constructed as follows. Consider the unit $m$-dimensional polydisc:
\[
  \mathbb{D}^m=\{(z_1,\ldots,z_m)\in \mathbb{C}^m\colon |z_i|^2 \le 1\text{ for }i=1,\ldots,m \}.
\]
Then 
\[
 \mathcal{Z}_{\mathcal{K}}:= \bigcup\limits_{I \in \mathcal{K}} \,\Bigl( \prod_{i \in I} \mathbb{D} \times \prod_{i    \notin I} \mathbb{S}\Bigr)\subset\mathbb{D}^m,
\] 
where $\mathbb{S}$ is the boundary of the unit disk $\mathbb{D}$.

The moment-angle complex is equipped with 
a natural action of the torus 
\[
  T^m=\{(t_1,\ldots,t_m)\in\C^m\colon |t_i|=1\}.
\]
When $\SK$ is simplicial subdivision of a sphere, $\momang$ is a topological manifold~\cite[Theorem~4.1.4]{buchstaber2015toric}, called the \emph{moment-angle manifold}.

We define an open submanifold $U(\SK)\subset\mathbb C^m$ in a similar way:
\[
  U(\SK):=\bigcup\limits_{I \in \mathcal{K}} \,
  \Bigl( \prod_{i \in I} \mathbb{C} \times 
  \prod_{i\notin I} \mathbb{C}^\times\Bigr),
\] 
where $\C^\times=\C\setminus\{0\}$.
The manifold $U(\SK)$ has a coordinate-wise action of the algebraic torus $(\mathbb{C}^{\times})^m$, in which $T^m$ is a maximal compact subgroup. Furthermore, 
$U(\SK)$ is a toric variety with the corresponding fan given by
\begin{equation*}
	\Sigma_\SK = \{ \R_\ge\langle\mb e_i \colon i \in I\rangle \colon I \in \SK\}, 
\end{equation*}
where $\mb e_i$ denotes the $i$-th standard basis vector of $\R^m$ and $\R_\ge\langle A\rangle$ denotes the cone spanned by the elements in $A$. 

By definition, $\momang$ is a compact subset of $U(\SK)$. Furthermore, $\momang$ is a $T^m$-equivariant deformation retract of~$U(\SK)$~\cite[Theorem~4.7.5]{buchstaber2015toric}. In the case when $\SK$ is the underlying complex of a complete simplicial fan, the deformation retraction $U(\SK)\to\momang$ can be realised as the projection onto the orbit space of a smooth free and proper action of a non-compact subgroup $R\subset(\C^\times)^m$ isomorphic 
to~$\R^{m-n}$, as described next.

\begin{constr}\label{realfol}
Suppose $\SK$ is the underlying complex of a complete simplicial (not necessarily rational) fan $\Sigma$ in an $n$-dimensional space~$V$. Choose generators $a_1,\ldots,a_m$ of the one-dimensional cones of~$\Sigma$ and consider the surjective linear map 
\begin{equation}\label{qmap}
  q\colon \R^m\to V,\quad e_i\mapsto a_i.
\end{equation}
Set
\begin{equation}\label{realsubgr}
  \mathfrak h'=\Ker q, \quad 
  R=\exp(\mathfrak h')=\{e^{h'}\colon h'\in\mathfrak h'\}\subset(\R^\times)^m,\quad
  H'=\exp(i\mathfrak h')
  \subset T^m.
\end{equation}
Note that the subgroup $H'\subset T^m$ is not closed unless $\mathfrak h'$ is a rational subspace of~$\R^m$.

By~\cite[Theorem~2.2]{panov2012complex}, the action of $R$ on $U(\SK)$ free and proper, and the quotient $U(\SK)/R$ is $T^m$-equivariantly homeomorphic to the moment-angle manifold 
$\momang$.
Furthermore, the subgroup $H'\subset T^m$ acts on $\momang=U(\SK)/R$ by restriction. The $H'$-action on $\momang$ is almost free (i.\,e., all stabiliser subgroups are discrete), see~\cite[Proposition~5.4.6]{buchstaber2015toric}. We therefore obtain a smooth foliation of $\momang$ by the orbits of~$H'$. This foliation and its holomorphic analogue will be studied further in Sections~\ref{basiczk} and~\ref{transverse}. 
\end{constr}

\subsection{Basic cohomology and equivariant cohomology} 
Let $\textgoth{g}$ be a Lie algebra. A \emph{$\g^\star$-differential graded algebra} ($\g^\star$-DGA for short) is a differential graded algebra (DGA for short) equipped with an action of operators $\iota_\xi$ (concatenation) and $L_\xi$ (Lie derivative) for $\xi \in \g$, see \cite[Definition 3.1]{goertsches2010equivariantbasic} for the details. For a $\g^\star$-DGA $(A, d_A)$, the basic subcomplex $A_{\mathrm{bas}\,\g}$ is given by 
\begin{equation*}
	A_\mathrm{bas\,\g} := \{ \omega \in A \colon \iota _\xi \omega = L_\xi \omega = 0 \text{ for any $\xi \in \g$} \}.
\end{equation*}
\emph{Basic cohomology} of $A$ is given by 
\[
  H_{\mathrm{bas}\,\g}(A) = H(A_{\mathrm{bas}\,\g},d_A). 
\]
We omit $\g$ by writing $H_{\mathrm{bas}}(A)$ for simplicity when $\g$ is clear from the context. 

Let $S(\g^*)$ denote the symmetric (polynomial) algebra on the dual Lie algebra $\g^*$ with generators of degree~$2$, and $\varLambda(\g^*)$  the exterior algebra with generators of degree~$1$. 
The \emph{Weil algebra} of $\g$ is the DGA
\[
  \W (\g):=\bigl(\varLambda(\g^*)\otimes S(\g^*),d_{\W(\g)}\bigr)
\]
with the standard acyclic (Koszul) differential~$d_{\W(\g)}$. 
We refer to $\W(\g)$ simply as $\W$ when $\g$ is clear from the context.
There are two models for equivariant cohomology of $A$. The \emph{Cartan model} is defined as 
\begin{equation*}
	\CM_{\g}(A) = ((S(\g^*) \otimes A)^\g, d_\g),  
\end{equation*}
where $(S(\g^*) \otimes A)^\g$ denotes the $\g$-invariant subalgebra.
We think of an element $\omega \in \CM_\g(A)$ as a $\g$-equivariant polynomial map from $\g$ to $A$. The differential $d_\g$ is given by 
\begin{equation*}
	d_\g(\omega)(\xi) = d_A(\omega(\xi)) - \iota_\xi(\omega(\xi)). 
\end{equation*}
The \emph{Weil model} is defined as 
\begin{equation*}
	\W_\g(A) = ((\W \otimes A)_\mathrm{bas}, d),
\end{equation*}
where $d = d_\W \otimes 1 + 1 \otimes d_A$. The Mathai--Quillen isomorphism~\cite{mathai1986superconnections} implies that the Weil model $W_\g(A)$ and the Cartan model $\CM_\g(A)$ have the same cohomology $H_\g(A)$. The algebra $H_\g(A)$ is called the \emph{$\g$-equivariant cohomology} of the $\g^\star$-algebra $A$. 

A \emph{$\W^\star$-algebra} $B$ is a $\g^\star$-DGA which is also a $\W$-module, see \cite[Definition 3.4.1]{guillemin2013supersymmetry}. For a $\W^\star$-algebra $B$, there are weak equivalences between $B_{\mathrm{bas}}$ and the algebras  $\CM_\g(B)$ and $\W_\g(B)$, see \cite[Section 5.1]{guillemin2013supersymmetry}. In particular, we have $H_\mathrm{bas}(B) \cong H_\g(B)$ if $B$ is a $\W^\star$-algebra. 

\smallskip

Now let $M$ be a smooth manifold equipped with an action of a connected Lie group $G$, and let $\g$ be the Lie algebra of~$G$. Then the algebra $\varOmega(M)$ of differential forms on $M$ is a $\W^\star$-algebra, so we have algebra isomorphisms
\[
  H_\mathrm{bas}(\varOmega(M))\cong H\bigl(\CM_{\g}(\varOmega(M))\bigr)\cong 
  H\bigl(\W_\g(\varOmega(M))\bigr).
\]
If in addition $G$ is a compact, then the algebra above is isomorphic to the equivariant cohomology $H_G^*(M) := H^*(EG \times _G M)$, see~\cite[Theorem~2.5.1]{guillemin2013supersymmetry}:
\[
H_\mathrm{bas}(\varOmega(M)) \cong H_G^*(M).
\]

\subsection{The case of torus actions} 
Let $G$ be a compact torus, and $M$ a smooth $G$-manifold. Let $\h'$ be a subspace of the Lie algebra $\g$ of $G$, and $H'$ the corresponding Lie subgroup of~$G$. Assume that the action restricted to $H'$ is almost free. Then we have a smooth foliation of $M$ by $H'$-orbits. It follows from \cite[Lemma 4.4]{ishida2018transverse} that $\varOmega(M)$ and $\varOmega(M)^G$ have a structure of $\W(\h')^\star$-algebras.  
\begin{lemma}\label{invariant}
	The natural inclusion $\varOmega(M)_{\mathrm{bas}\,\h'}^G \hookrightarrow \varOmega(M)_{\mathrm{bas}\,\h'}$ is a quasi-isomorphism.
\end{lemma}
\begin{proof}
First, we show that the induced homomorphism $H_{\mathrm{bas}\,\h'}(\varOmega(M)^G) \to H_{\mathrm{bas}\,\h'}(\varOmega(M))$ is injective. Let $I \colon\varOmega(M)_{\mathrm{bas}\,\h'} \to \varOmega(M)^G_{\mathrm{bas}\,\h'}$ be the linear map given by 
		\begin{equation*}
			I(\alpha) = \int_{g \in G} g^*\alpha\, dg,\quad \alpha \in \varOmega(M)_{\mathrm{bas}\,\h'},
		\end{equation*}
		where $dg$ denotes the normalised Haar measure on $G$. Then the composite $\varOmega(M)^G_{\mathrm{bas}\,\h'} \hookrightarrow \varOmega(M)_{\mathrm{bas}\,\h'}\stackrel I\longrightarrow \varOmega(M)^G_{\mathrm{bas}\,\h'}$ is the identity. Passing to cohomology, we obtain that the homomorphism $H_{\mathrm{bas}\,\h'}(\Omega(M)^G) \to H_{\mathrm{bas}\,\h'}(\Omega(M))$ is injective. 
		
Now, we prove that $H_{\mathrm{bas}\,\h'}(\varOmega(M)^G) \to H_{\mathrm{bas}\,\h'}(\varOmega(M))$ is surjective. This will be done by showing that $I(\alpha)$ and $\alpha$ define the same cohomology class in $H_{\mathrm{bas}\,\h'}(\varOmega(M))$. 

Let $[\alpha] \in H_{\mathrm{bas}\,\h'}(\varOmega(M))$ be a cohomology class represented by $\alpha \in \varOmega_{\mathrm{bas}\,\h'}(M)$. Let $\exp_G \colon \g \to G$ be the exponential map, and let $\gamma_1,\dots, \gamma_n$ be a lattice basis of $\Ker \exp_G$. Define 
		\begin{equation*}
			D:= \biggl\{ v = \sum_{i=1}^n a_i \gamma_i :  0\leq a_i <1\biggr\}.
		\end{equation*} 
		Then the exponential map restricted to $D$ gives a bijection $\exp_G|_D \colon D \to G$. For $v \in D$ and $t \in \R$, we define $g_t := \exp_G(tv) \in G$ and 
		\begin{equation*}
			\theta_{g_1} := \int_0^1 g_t^*(\iota_{X_v}\alpha) dt \in \varOmega_{\mathrm{bas}\,\h'}(M).
		\end{equation*}
The form $\theta_{g_1}$ is basic because $G$ is a commutative group.
We have
		\begin{equation*}
			\lim_{h \to 0} \frac{g^*_{t+h}\alpha - g^*_t\alpha}{h} = L_{X_v}g^*_t\alpha
				= i_{X_v}dg^*_t\alpha + di_{X_v}g^*_t\alpha 
				= di_{X_v}g^*_t\alpha 
				= dg^*_ti_{X_v}\alpha, 
		\end{equation*}
implying that
		\begin{equation*}
			d\theta_{g_1} = d\int_0^1 g_t^*i_{X_v}\alpha dt 
				=\int_0^1 dg_t^*i_{X_v}\alpha dt
				= g^*_1\alpha - g^*_0 \alpha 
				= g^*_1 \alpha - \alpha.
		\end{equation*}
Therefore 
		\begin{equation*}
				\int_{g_1 \in G}(g^*_1\alpha-\alpha) dg = \int_{g_1 \in G}d \theta_{g_1}dg 
				= d\int_{g_1 \in G}\theta_{g_1} dg. 
		\end{equation*}
On the other hand, 
		\begin{equation*}
			\int_{g_1 \in G}(g^*_1\alpha-\alpha) dg = I(\alpha) - \alpha.
		\end{equation*}
It follows that $I(\alpha)$ and $\alpha$ represent the same class in $H_{\mathrm{bas}\,\h'}(\varOmega(M))$. This together with $I(\alpha) \in \varOmega_{\mathrm{bas}\,\h'}(M)^G$ yields that the induced homomorphism $H_{\mathrm{bas}\,\h'}(\varOmega(M)^G)\to H_{\mathrm{bas}\,\h'}(\varOmega(M))$ is surjective, proving the lemma.
\end{proof}

\section{Basic cohomology of $\momang$}\label{basiczk}
Here we describe the basic cohomology algebra of $\momang$ with respect to the foliation by the orbits of the subgroup $H'\subset T^m$ defined in Construction~\ref{realfol}:
\[
  \hb^*(\momang) := H_{\mathrm{bas}\,\h'}(\varOmega (\momang)).
\]  
We first reduce the computation of $\hb^*(\mathcal Z_\Sigma)$ to cohomology of a special DGA:

\begin{lemma}\label{cohN}
Consider the algebra
\[
  \NB:=\CM_{\h' }\bigl(\varOmega (\momang)^{T^m}\bigr)=
  \bigl(S(\h'^*)\otimes\varOmega (\momang)^{T^m}
  ,d_{\h'}\bigr). 
\]
Then we have an isomorphism
\[
  \hb^*(\momang)\cong H(\NB).
\]
\end{lemma}
\begin{proof}
Applying the quasi-isomorphism of Lemma \ref{invariant} to the $\W(\h')^\star$-algebras $\varOmega (\momang)$ and $\varOmega (\momang)^{T^m}$ we obtain
\[
  \hb^*(\momang)=H_{\mathrm{bas}\,\h'}\bigl(\varOmega(\momang)\bigr)\cong
  H_{\mathrm{bas}\,\h'}\bigl(\varOmega(\momang)^{T^m}\bigr)=
  H\bigl(\CM_{\h' }\bigl(\varOmega (\momang)^{T^m}\bigr)\bigr)=H(\mathcal N).\qedhere
\]
\end{proof}

Recall that a DGA $B$ is called {\it formal} if it is weak equivalent to its cohomology algebra: $(B, d_B) \simeq (H^*(B, d_B), 0)$. (A \emph{weak equivalence} is the equivalence relation generated by quasi-isomorphisms; it may be realised not by a single quasi-isomorphism of DGAs, but rather by a zigzag of quasi-isomorphisms.)

As $T^m$ is compact, cohomology of the Cartan model
\[
  \CM_{\textgoth{t}}(\varOmega(\momang))=\bigl(S(\textgoth{t}^*)\otimes
  \varOmega(\momang)^{T^m},\,d_{\textgoth{t}}\bigr)
\]
is the equivariant cohomology~$H^*_{T^m}(\momang)$, which is a module over $S(\textgoth{t}^*)=H^*_{T^m}(\mathit{pt})=H^*(BT^m)$.

\begin{lemma} \label{formality}
The algebra $\CM_{\textgoth{t}}(\varOmega(\momang))$ is formal. Furthermore, there is a zigzag of quasi-isomorphisms of DGAs between 
$\CM_{\textgoth{t}}(\varOmega(\momang))$ and $H_{T^m}(\momang)$ which respect the $S(\textgoth{t}^*)$-module structure.
\end{lemma}

\begin{proof}
In this proof, $\W$ is the Weil algebra $\W(\textgoth{t})$ of the torus $T^m$.
Let $E=\mathit{EU}(m)$ be the space of orthonormal $m$-frames in $\mathbb{C}^{\infty}$. Let $\varOmega(E)$ be the inverse limit of the algebras of differential forms on the smooth manifolds of $m$-frames in $\C^N$.
We consider the commutative diagram
\[
\begin{tikzcd}
     & \FoMin \otimes \W\arrow[hookrightarrow]{r}[swap]{\iota} & \FoMin \otimes \varOmega({E}) & \\
     \CM_{\textgoth{t}}(\FoM) & (\FoMin \otimes \W)_{\mathrm{bas}}\arrow{l}{\varphi}[swap]{\simeq}\arrow[hookrightarrow]{u}\arrow[hookrightarrow]{r}{\simeq}[swap]{\iota_{\mathrm{bas}}} & (\FoMin \otimes \varOmega({E}))_{\mathrm{bas}}\arrow{r}{\cong}[swap]{\psi}\arrow[hookrightarrow]{u} & \varOmega(\momang \times_{T^m} {E}) \\
     S(\textgoth{t}^*)\arrow[hookrightarrow]{u}\arrow{r}{\cong} &  \W_{\mathrm{bas}}\arrow[hookrightarrow]{u}\arrow[hookrightarrow]{r}{\simeq} & \varOmega({E})_{\mathrm{bas}}\arrow[hookrightarrow]{u}\arrow{r}{\cong} & \varOmega(BT^m)\arrow[hookrightarrow]{u}
\end{tikzcd}
\]
Here $\iota$ and $\iota_{\mathrm{bas}}$ are the quasi-isomorphisms induced by the inclusion $\W\hookrightarrow\varOmega( E)$ of a free acyclic $\W^\star$-algebra (see~\cite[Proposition~2.5.4 and \S4.4]{guillemin2013supersymmetry}), and the restriction $\W_{\mathrm{bas}}\hookrightarrow\varOmega(E)_{\mathrm{bas}}$ is the Chern--Weil homomorphism. The quasi-isomorphism $\varphi$ is given by Cartan's Theorem (see~\cite[Theorem~4.2.1]{guillemin2013supersymmetry}). The isomorphism $\psi$ follows from the fact that $T^m$ acts freely on~$E$.

The middle line of the diagram above gives a zigzag of quasi-isomorphisms between $\CM_{\textgoth{t}}(\FoM)$ and $\varOmega(\momang \times_{T^m} {E})$ which respect the $S(\textgoth{t}^*)$-module structure.

Now, the Borel construction $\momang \times_{T^m} {E}$ is homotopy equivalent to the polyhedral product $(\mathbb{C}P^{\infty})^{\SK}$, which is a rationally formal space by \cite[Theorem~4.8]{notbohm2005davis} or~\cite[Theorem~8.1.6, Corollary~8.1.7]{buchstaber2015toric}. Rational formality implies a zigzag of quasi-isomorphisms between $\varOmega(\momang \times_{T^m} {E})$ and $H^*_{T^m}(\momang)=H^*(\momang \times_{T^m} {E})$, as the de Rham forms $\varOmega(\momang \times_{T^m} {E})$ is a commutative cochain model. This zigzag can be chosen to respect the $H^*(BT^m)$-module structure (see~\cite[8.1.11--8.1.12]{buchstaber2015toric}).
\end{proof}

We have the following extended functoriality property of $\Tor$ in the category of DGAs, which is a standard corollary of the Eilenberg-Moore spectral sequence:

\begin{lemma}[{\cite[Corollary 1.3]{smith1967homological}}]\label{Eilenberg-Moore}
Let $A$ and $B$ be DGAs, let $L,L'$ be a pair of $A$-modules and let $M,M'$ be a pair of $B$-modules given together with morphisms 
\[
  f\colon A\to B,\quad g\colon L\to M,\quad g'\colon L'\to M'
\]
where $g$ and $g'$ are $f$-linear. If $f$, $g$ and $g'$ are quasi-isomorphisms, then
\[
  \Tor_f(g,g')\colon\Tor_A(L,L')\to\Tor_B(M,M')
\]
is an isomorphism.
\end{lemma}

Given a commutative ring $R$ with unit, the \emph{face ring} (or the \emph{Stanley--Reisner ring}) of a simplicial complex $\SK$ is
\[
 R[\SK]:=R[v_1,...,v_m]/I_{\SK},
\]
where $R[v_1,...,v_m]$ is the polynomial algebra, and $I_{\SK}$ is the \emph{Stanley--Reisner ideal}, generated by those square-free monomials $v_I=\prod_{i\in I}v_i$ for which $I\not\in\SK$.  We make 
$R[\SK]$ a graded ring by setting $\deg v_i=2$ for $i=1,\ldots,m$.

Now we are ready to prove the main result of this section:

\begin{thm}\label{basicTh}
There is an isomorphism of algebras:
\[
  \hb^*(\momang) \cong \mathbb{R}[v_1,\ldots,v_m] /(I_{\SK}+J),
\]
where $I_\SK$ is the Stanley--Reisner ideal of~$\SK$, generated by the monomials 
\[
  v_{i_1}\cdots v_{i_k}\quad\text{with }\{i_1,\ldots,i_k\}\notin\SK,
\]  
and $J$ is the ideal generated by the linear forms
\[
  \sum_{i=1}^m \langle \mb u, q(\mb e_i) \rangle v_i\quad
  \text{with }\mb u\in (\textgoth{t}/\textgoth{h}')^*.
\]
Here $q\colon\mathfrak t\to\mathfrak  t/ \h'$ is the projection, and $\mathfrak t=\R^m$.
\end{thm}

\begin{proof}
Denote $\textgoth{g}':=\textgoth{t}/\textgoth{h}'$. We have a splitting $\textgoth{t}\cong\textgoth{g}'\oplus \textgoth{h}'$. Hence, $S({\textgoth{t}}^*)\cong S(\textgoth{g}'^*)\otimes S(\textgoth{h}'^*)$, and $S({\textgoth{t}}^*)$ is an $S(\textgoth{g}'^*)$-module via the linear monomorphism $q^*\colon \textgoth{g}'^*\to {\textgoth{t}}^*$. We also obtain a DGA isomorphism 
\begin{equation}\label{CTmsplit}
  \CM_{\textgoth{t}}(\FoM) \cong S(\textgoth{g}'^*) \otimes \NB,
\end{equation}  
where $\NB=\CM_{\h' }\bigl(\varOmega (\momang)^{T^m}\bigr)$ (see Lemma~\ref{cohN}) and the right hand side is understood as the Cartan model of $\NB$ with respect to the Lie algebra $\textgoth{g}'$.

Recall that $H^*_{T^m}(\momang) \cong \mathbb{R}[\SK]$ (see \cite[Corollary~3.3.1]{buchstaber2000torus}). Since the fan $\Sigma = q(\Sigma_\SK)$ is complete, the Stanley--Reisner ring $\R[\SK]$ is Cohen--Macaulay, that is, it is a finitely-generated free module over its polynomial subalgebra. Furthermore, the composite $\textgoth{g}'^*\hookrightarrow {\textgoth{t}}^*\to\textgoth{t}^*_I$ is onto for any $I\in\SK$, where $\textgoth{t}_I$ is the coordinate subspace generated by all $\mb e_i$ with $i \in I$. Therefore, the criterion \cite[Lemma 3.3.1]{buchstaber2015toric} applies to show that $\mathbb{R}[\SK]$ is a finitely generated free module over $S(\textgoth{g}'^*)$.

Consider the following pushout diagram of DGAs:
\[
\begin{tikzpicture}
  \matrix (m) [matrix of math nodes,row sep=3em,column sep=4em,minimum width=2em]
  {
     \NB \cong \mathbb{R} \otimes_{S^*(\textgoth{g}'^*)} \CM_{\textgoth{t}}(\FoM) \, \, \, & \, \, \, \, \, \, \CM_{\textgoth{t}}(\FoM) \cong S(\textgoth{g}'^*)   \otimes \NB  \\
     \mathbb{R} & S(\textgoth{g}'^*) \\};
  \path[-stealth]
    (m-2-2) edge node [below] {$f^*$} (m-2-1)
            edge  node [right] {$\pi^*$} (m-1-2)
    (m-2-1) edge node [left] {$\tilde{\pi}^*$} (m-1-1)
    (m-1-2) edge node [above] {$\tilde{f}^*$} (m-1-1);
  \end{tikzpicture}
\]
where the morphisms are given by
\[ 
f^*\colon p \mapsto p(0),\quad
\pi^*\colon p \mapsto p \otimes 1,\quad
\tilde{f}^*\colon \omega \mapsto 1 \otimes \omega,\quad
\tilde{\pi}^*\colon c \mapsto c \otimes 1.
\]
We have a sequence of algebra isomorphisms:
\begin{multline}\label{TorR}
  \Tor_{S(\textgoth{g}'^*)}\bigl(\mathbb{R},S^*(\textgoth{g}'^*)\otimes \NB\bigr)
  \cong\Tor_{S(\textgoth{g}'^*)}\bigl(\mathbb{R},\CM_{\textgoth{t}}(\FoM)\bigr)
  \cong\Tor_{S(\textgoth{g}'^*)}\bigl(\mathbb{R},H^*_{T^m}(\momang)\bigr)\\  
  \cong \Tor^0_{S(\textgoth{g}'^*)}\bigl(\mathbb{R},H^*_{T^m}(\momang)\bigr)
  \cong \R \otimes_{S(\textgoth{g}'^*)} H^*_{T^m}(\momang) \cong H^*_{T^m}(\momang) / S^+(\textgoth{g}'^*)\\     \cong \mathbb{R}[v_1,...,v_m] /(I_{\SK}+J).
\end{multline}
The first isomorphism follows from~\eqref{CTmsplit}.
The second isomorphism follows from Lemma~\ref{formality} and Lemma~\ref{Eilenberg-Moore}. In the third isomorphism, the higher $\Tor$ vanish because $H^*_{T^m}(\momang)$ is a free module over $S^*(\textgoth{g}'^*)$. The fourth and fifth isomorphisms are clear. For the last isomorphism, recall that $q \colon\textgoth{t}\to\textgoth{t}/\textgoth{h}'=\textgoth{g}'$ is the quotient projection, so that 
$q^*(\mb u)=\sum_{i=1}^m \langle \mb u,\mb a_i \rangle v_i$ for any $\mb u\in\textgoth{g}'^*$.

On the other hand, we have a sequence of isomorphisms
\begin{multline}\label{TorL}
  \Tor_{S(\textgoth{g}'^*)}\bigl(\mathbb{R}, S(\textgoth{g}'^*) \otimes\NB\bigr) 
  \cong \Tor_{S(\textgoth{g}'^*)}^0\bigl(\mathbb{R}, S(\textgoth{g}'^*) \otimes\NB\bigr)  
  \cong H(\mathbb{R}) \otimes_{S(\textgoth{g}'^*)} H\bigl(S(\textgoth{g}'^*) \otimes \NB\bigr) \\
  \cong H\bigl(\R \otimes_{S(\textgoth{g}'^*)} (S(\textgoth{g}'^*) \otimes \NB)\bigr)
  \cong H(\NB)\cong \hb^*(\momang).
\end{multline}
For the first isomorphism, the higher $\Tor$ vanish by~\eqref{TorR}. The second isomorphism is by definition of $\Tor^0$. The third isomorphism follows from the K\"unneth Theorem, since
$H\bigl(S(\textgoth{g}'^*) \otimes \NB\bigr)=H^*_{T^m}(\momang)$ is a free module over $S^*(\textgoth{g}'^*)$.
The fourth isomorphism is clear. The last isomorphism is Lemma~\ref{cohN}.

The theorem follows from~\eqref{TorR} and~\eqref{TorL}. 
\end{proof}

\section{Complex moment-angle manifolds and maximal torus actions}

We consider moment-angle manifolds $\momang$ with a $T^m$-invariant complex structure. The necessary and sufficient conditions for the existence of such a structure were established in~\cite{panov2012complex} and~\cite{ishida2013complex}. Namely, there is the following holomorphic version of Construction~\ref{realfol} defining a complex structure and a holomorphic foliation 
on~$\momang$.

\begin{constr}[Complex structure on~$\momang$]\label{czk}
Assume that $\dim\momang=m+n$ is even; this can always be achieved by adding ghost vertices to~$\SK$.
A $T^m$-invariant complex structure on $\momang$ is defined by two pieces of data:
\begin{itemize}
\item[--] a complete simplicial fan $\Sigma=\{\SK;a_1,\ldots,a_m\}$ in $V$ with underlying simplicial complex $\SK$ and fixed generators $a_1,\ldots,a_m$ of one-dimensional cones (a \emph{marked fan});

\item[--] a choice of a complex structure on the kernel of the linear map~$q\colon\R^m\to V$, see~\eqref{qmap}.
\end{itemize}

A choice of a complex structure on $\Ker q$ is equivalent to a choice of an $\frac{m-n}2$-dimensional complex subspace $\mathfrak h\subset\C^m$ satisfying the two conditions:
\begin{itemize}
	\item[(a)] the composite $\h\hookrightarrow\C^m\stackrel{\mathrm{Re}}\longrightarrow \R^m$ is injective; 
	\item[(b)] the composite
	$\h\hookrightarrow\C^m\stackrel{\mathrm{Re}}\longrightarrow \R^m
	\stackrel q\longrightarrow V$ is zero. 
\end{itemize}
Consider the $\frac{m-n}2$-dimensional complex-analytic subgroup
\[
  H=\exp(\mathfrak h)\subset(\C^\times)^m.
\]
By \cite[Theorem~3.3]{panov2012complex}, the holomorphic action of $H$ on $U(\SK)$ is free and proper, and the complex manifold $U(\SK)/H$ is $T^m$-equivariantly homeomorphic to~$\momang$. 

Conversely, assume that a moment-angle manifold $\momang$ admits a $T^m$-invariant complex structure. By~\cite[Theorem 7.9]{ishida2013complex}, the manifold $\momang$ is $T^m$-equivariantly biholomorphic to the quotient $U(\SK)/H$ as above. The marked fan $\Sigma$ and the complex subspace $\mathfrak h\subset\C^m$ are recovered as follows. The action of $T^m$ on $\momang$ extends to a holomorphic action of $(\C^\times)^m$ on~$\momang$, although the latter action is not effective. The global stabilisers subgroup (the noneffectivity kernel)
\[
  H= \{ g \in (\C^\times)^m \colon g\cdot x = x \text{ for all $x \in \momang$}\}
\]  
is a complex-analytic subgroup of $(\C^\times)^m$. The Lie algebra $\h$ of $H$ is a complex subalgebra of the Lie algebra $\C^m$ of $(\C^\times)^m$. 
By \cite[Proposition 7.8]{ishida2013complex}, it satisfies the following:
\begin{itemize}
	\item[(a)] the composite $\h\hookrightarrow\C^m\stackrel{\mathrm{Re}}\longrightarrow \R^m$ is injective; 
	\item[(b)] the quotient map $q \colon \R^m \to \R^m/\mathop{\mathrm{Re}}(\h)$ sends the fan $\Sigma_\SK$ to a complete fan $\Sigma=q(\Sigma_\SK)$ in $\R^m/\mathop{\mathrm{Re}}(\h)$. 
\end{itemize}
Here we identify $\R^m$ with the Lie algebra $\mathfrak{t}$ of~$T^m$.

It follows that a moment-angle manifold $\momang$ admits a complex structure if and only if $\SK$ is the underlying complex of a complete simplicial fan (that is, $\SK$ is a \emph{star-shaped} sphere triangulation), and a stably complex structure on such $\momang$ is defined by a choice of a complex subspace $\mathfrak h\subset\C^m$ satisfying (a) and (b) above.
\end{constr}

Now we proceed to describe the canonical holomorphic foliation on~$\momang$. 

\begin{constr}[Holomorphic foliation on $\momang$]\label{canfol}
Recall the Lie subgroup $H'\subset T^m$ and its Lie algebra $\mathfrak h'\subset \R^m= \mathfrak{t}$ from~\eqref{realsubgr}. Observe that $\mathfrak h'=\Ker q=\mathop{\mathrm{Re}}(\mathfrak h)$.
The complexification $(\mathfrak h')^\C\subset\C^m$ is $\Ker q^\C$, where $q^\C\colon\C^m\to V^\C$ is the complexification of~\eqref{qmap}.
Now define the complex $(m-n)$-dimensional Lie group
\begin{equation*}
  (H')^{\mathbb{C}}=\exp((\mathfrak h')^\C)=\exp(\Ker q^\C)\subset(\C^\times)^m.
\end{equation*}
The restriction of the $(\C^\times)^m$-action on $U(\SK)$ to $(H')^{\mathbb{C}}$ has discrete stabilisers (see~\cite[Proposition~4.2]{p-u-v16} or \cite[Propositions 3.3 and 5.2]{ishida2017torus}. We therefore obtain a holomorphic foliation of $U(\SK)$ by the orbits 
of~$(H')^{\mathbb{C}}$. The holomorphic foliation $\mathcal F_\Sigma$ is mapped by the quotient projection $U(\SK)\to U(\SK)/H$ to a holomorphic foliation of $\momang\cong U(\SK)/H$ by the orbits of $(H')^\C/H\cong H'$. We denote the latter foliation by $F_{\h'}$ and refer to it as the \emph{canonical holomorphic foliation} of~$\momang$.
\end{constr}


\begin{rmk}\label{holofib}
If the subspace $\textgoth{h}' \subset \R^m$ is rational (i.\,e., generated by integer vectors),  then $H'$ is a subtorus of $T^m$ and the complete simplicial fan $\Sigma=q(\Sigma_\SK)$ is rational. 
The rational fan $\Sigma$ defines a toric variety $V_{\Sigma}=\momang/H'=U(\SK)/(H')^{\C}$.
The holomorphic foliation of $\momang$ by the orbits of $H'$ becomes a holomorphic Seifert \emph{fibration} over the toric orbifold $V_\Sigma$ with fibres compact complex tori $(H')^\C/H$ (see~\cite[Proposition~5.2]{panov2012complex}).
\end{rmk}

Complex moment-angle manifolds $\momang$ are example of \emph{complex manifolds with maximal torus action}, which we review following~\cite{ishida2013complex}.
 
Let $M$ be a connected smooth manifold equipped with an effective action of a compact torus $G$. We say that the $G$-action on $M$ is \emph{maximal} if there exists a point $x \in M$ such that $\dim G+\dim G_x = \dim M$. 
If the action of $G$ on $M$ is maximal, then we can think of $G$ as a maximal compact torus of the group of diffeomorphisms on $M$ (see \cite[Lemma 2.2]{ishida2013complex}). Examples of maximal torus actions include the half-dimensional torus action on a smooth toric variety and the $T^m$-action on a moment-angle manifold~$\momang$.
	 
Let $\mathscr{C}_1$ denote the \emph{category of complex manifolds with maximal torus actions}, with objects given by triples $(M,G,y)$, where
	\begin{itemize}
	 	\item $M$ is a compact connected complex manifold;
		\item $G$ is a compact torus acting on $M$, the $G$-action is maximal and preserves the complex structure on $M$; 
		\item $y \in M$ satisfies $G_y = \{1\}$. 
	\end{itemize}
The set of morphisms $\Hom_{\mathscr{C}_1} ((M_1,G_1,y_1), (M_2,G_2,y_2))$ consists of pairs $(f, \alpha)$, where
	\begin{itemize}
	 	\item $\alpha \colon G_1 \to G_2$ is a smooth homomorphism; 
		\item $f\colon M_1\to M_2$ is an $\alpha$-equivariant holomorphic map, i.\,e. $f(g\cdot x) = \alpha(g)\cdot f(x)$  for $x \in M_1$ and $g\in G_1$; 
		\item $f(y_1) = y_2$. 
	\end{itemize}

Given a compact torus $G$, we denote by $\g$ the Lie algebra of $G$ and by $\exp_G \colon\g \to G$ the exponential map. We think of $\Ker \exp_G\subset \g$ as a lattice in~$\g$. Let $\g^\C = \g \otimes_\R \C=\g\oplus i\g$ be the complexified Lie algebra. We denote by
$p \colon \g^\C\to\g$ the first projection.

As a combinatorial counterpart of $\mathscr{C}_1$, we consider the category $\mathscr{C}_2$ with objects given by triples $(\Sigma, \h, G)$ satisfying the following:
\begin{itemize}
	 	\item $G$ is a compact torus;
		\item $\Sigma$ is a nonsingular fan in $\g$ with respect to the lattice $\Ker \exp_G$;
        \item $\h\subset\g^\C$ is a complex subspace such that the restriction $p|_{\h} \colon\h \to \g$ is injective;
    we denote by $q \colon \g \to \g/p(\h)$ be the quotient map of real vector spaces;
		\item $q(\Sigma) := \{ q(\sigma) \subset \g/p(\h)\colon 
        \sigma \in \Sigma\}$
		is a complete fan, and the map $\Sigma \to q(\Sigma)$ given by $\sigma \mapsto q(\sigma)$ is bijective. 
\end{itemize}
The morphisms $\Hom _{\mathscr{C}_2} ((\Sigma_1,\h_1,G_1), (\Sigma_2,\h_2,G_2))$ are defined as the set of smooth homomorphisms $\alpha \colon G_1 \to G_2$ with the following properties:
\begin{itemize}
	  	\item the differential $d\alpha \colon\g_1 \to \g_2$ induces a morphism of fans $\Sigma_1\to\Sigma_2$ (that is, for any $\sigma_1 \in \Sigma_1$, there exists $\sigma_2 \in \Sigma_2$ such that $d\alpha(\sigma_1) \subset \sigma_2$); 
		\item the complexified differential $d\alpha^\C \colon \g_1 ^\C \to \g_2^\C$ satisfies $d\alpha^\C (\h_1) \subset \h_2$. 
\end{itemize}
It is proved in \cite[Theorem 8.2]{ishida2013complex} that
the categories $\mathscr{C}_1$ and $\mathscr{C}_2$ are equivalent. 

Namely, there is a functor $\mathscr{F}_1\colon \mathscr{C}_1 \to \mathscr{C}_2$
defined as follows. For $(M,G,y) \in \mathscr{C}_1$, there exists a unique $(\Sigma, \h, G) \in \mathscr{C}_2$ such that $M$ is $G$-equivariantly biholomorphic to the quotient manifold $V_\Sigma/H$, where $V_\Sigma$ is the toric variety associated with $\Sigma$ and $H$ is the subgroup of the algebraic torus $G^\C$ corresponding to $\h \subset \g^\C$. The category $\mathscr{C}_1$ fully contains LVMB manifolds and, in particlar, moment-angle manifolds with invariant complex structures (see \cite[Section 10]{ishida2013complex} for the details). 

In the opposite direction, there is a functor $\mathscr{F}_2 \colon \mathscr{C}_2 \to \mathscr{C}_1$ defined as follows. Given $(\Sigma, \h, G) \in \mathscr{C}_2$, define the manifold $M$ as the quotient $V_\Sigma/H$ with the natural $G$-action. This gives an object in~$\mathscr{C}_1$.
In particular, if $\Sigma$ is a subfan of the standard fan in $\g=\R^m$ defining the toric variety $\C^m$, then the manifold  $V_\Sigma/H$ is $G$-equivariantly homeomorphic to the moment-angle manifold $\momang$, where $\SK$ is the underlying simplicial complex of $\Sigma$. If $\Sigma$ is a subfan of the fan defining the toric variety $\C P^m$, then the corresponding manifold $V_\Sigma/H$ is an LVMB manifold.

The functors $\mathscr{F}_1 \colon \mathscr{C}_1 \to \mathscr{C}_2$ and $\mathscr{F}_2 \colon \mathscr{C}_2 \to \mathscr{C}_1$ are weak inverse to each other. 
	
Given $(M,G,y) \in \mathscr{C}_1$, we define a holomorphic foliation on $M$ as in the case of moment-angle manifolds, see Construction~\ref{canfol}. Namely, we consider the Lie subgroup $H'$ of $G$ corresponding to the Lie subalgebra $\h' := p(\h) \subset \g$. It was shown in \cite[Propositions 3.3 and 5.2]{ishida2017torus} that the action of $H'\subset G$ on $M$ is almost free. Thus we have a holomorphic foliation of $M$ by $H'$-orbits. We refer to this foliation as the \emph{canonical foliation} on $M$ (see~\cite[Section~2]{ishida2018transverse}).

\section{Transverse equivalence}\label{transverse}
	Let $(M_1,F_1)$ and $(M_2,F_2)$ be smooth manifolds with foliations $F_1$ on $M_1$ and $F_2$ on $M_2$. We say that $(M_1,F_1)$ and $(M_2,F_2)$ are \emph{transversely equivalent} if there exist a foliated manifold $(M_0,F_0)$ and surjective submersions $f _i \colon M_0 \to M_i$ for $i=1,2$ such that 
	\begin{itemize}
		\item $f_i^{-1}(x_i)$ is connected for all $x_i \in M_i$, and
		\item the preimage under $f_i$ of every leaf of $F_i$ is a leaf of $F_0$
	\end{itemize}
(see \cite[Definition 2.1]{Molino} for details). 

The important property of the transverse equivalence is that the algebra of basic differential forms is an invariant of the equivalence class:

\begin{proposition}\label{bcohinv}
If foliated manifolds $(M_1,F_1)$, $(M_2,F_2)$ are transversely equivalent via $(M_0, F_0)$ and $f_i \colon M_0 \to M_i$, then there is a DGA isomorphism $\varOmega^*_{\mathrm{bas}}(M_1)\cong \varOmega^*_{\mathrm{bas}}(M_2)$.
\end{proposition}
\begin{proof}
We show that each $f_i^* \colon \varOmega^*_{\mathrm{bas}}(M_i) \to \varOmega^*_{\mathrm{bas}}(M_0)$ is a DGA isomorphism. The map $f_i^*$ is injective since $f_i$ is a submersion. To prove that $f_i^*$ is surjective, we take a basic form $\omega \in \varOmega^q_{\mathrm{bas}}(M_0)$ and construct $\omega' \in \varOmega^q_{\mathrm{bas}}(M_i)$ such that $f_i^*\omega'=\omega$. Choose a point $x_i\in M_i$ and $q$ tangent vectors $v_1,\ldots,v_q \in\mathcal TM_i|_{x_i}$. Take any $x_0 \in f_i^{-1}(x_i)$ and tangent vectors $u_1,\ldots,u_q \in \mathcal TM_0|_{x_0}$ such that $df_i(u_j)=v_j$. Then put $\omega'(v_1,\ldots,v_q):=\omega(u_1,\ldots,u_q)$. This definition is independent of all choices. First, it is independent of the choice of a point $x_0$, since $f^{-1}(x_i)$ is connected, belongs to a single leaf of $F_0$ and $L_\xi\omega=0$ for any section $\xi$ of $\mathcal TF_0$. Second, the definition of $\omega'$ is independent of the choice of vectors $u_1,\ldots,u_q$, since $\Ker df_i|_{x_0}\subset\mathcal TF_0|_{x_0}$ and $\omega$ is a basic form. Thus, $\omega'$ is a well-defined basic form and $f_i^*(\omega')=\omega$.
\end{proof}

Transverse equivalence is an equivalence relation on foliated manifolds. Restricting our attention to complex manifolds with maximal torus actions and their canonical foliations, we obtain the appropriate version of transverse equivalence, as described next. 

\begin{defn}\label{defn:equivalent}
Let $(M_1,G_1,y_1), (M_2,G_2,y_2) \in \mathscr{C}_1$. We say that triples $(M_1,G_1,y_1)$ and $(M_2,G_2,y_2)$ are \emph{principal equivalent (p-equivalent for short)} if there exist $(M_0,G_0,y_0) \in \mathscr{C}_1$ and morphisms $(f_i, \alpha_i) \in \Hom_{\mathscr{C}_1}((M_0,G_0,y_0), (M_i,G_i,y_i))$ for $i=1,2$ such that
		\begin{itemize}
			\item $\Ker \alpha_i$ is connected; 
			\item $f_i \colon M_0 \to M_i$ is a principal $\Ker \alpha_i$-bundle. 
		\end{itemize}
\end{defn}
    

\begin{lemma}\label{lemma:equivalence}
Let $(M,G,y), (M_0,G_0,y_0) \in \mathscr{C}_1$. Let $F$ and $F_0$ be the canonical foliations on $M$ and $M_0$, respectively. Let 
$
			(f,\alpha) \in \Hom_{\mathscr{C}_1}((M_0,G_0,y_0),(M,G,y))
$
be a morphism such that
		\begin{itemize}
			\item $\Ker \alpha$ is connected;
			\item $f \colon M_0 \to M$ is a principal $\Ker \alpha$-bundle. 
		\end{itemize}
Then, 
		\begin{itemize}
			\item $f^{-1}(x)$ is connected for all $x \in M$ and
			\item the preimage under $f$ of every leaf of $F$ is a leaf of $F_0$.
		\end{itemize}
	\end{lemma}
	\begin{proof}
		Since $\Ker \alpha$ is connected and $f \colon M_0 \to M$ is a principal $\Ker \alpha$-bundle, we have that $f^{-1}(x)$ is connected for all $x \in M$. 
		
		We put $(\Sigma, \h, G) = \mathscr{F}_1(M,G,y)$ and $(\Sigma_0,\h_0,G_0) = \mathscr{F}_1(M_0,G_0,y_0)$. 
		It follows from \cite[Theorem 11.1]{ishida2013complex} that $\alpha$ is surjective and the differential $d\alpha \colon \g_0 \to \g$ induces a one-to-one correspondence between the primitive generators of $1$-cones in $\Sigma$ and the primitive generators of $1$-cones in $\Sigma'$. 
		Let $p_0 \colon \g_0^\C \to \g_0$ and $p \colon \g^\C \to \g$ be the projections. We put $\h'=p(\h)$ and $\h'_0 = p_0(\h_0)$. Let $q_0 \colon \g_0 \to \g_0/\h_0'$ and $q \colon \g \to \g/\h'$ be the quotient maps. Since $d\alpha^\C (\h_0) \subset \h$ and $p\circ d\alpha^\C = d\alpha  \circ p_0$, we have $d\alpha(\h_0') \subset \h'$. Hence, $d\alpha \colon \g_0 \to \g$ induces a linear map $\overline{d\alpha} \colon \g_0/\h_0' \to \g/\h'$. Since $d\alpha$ induces a one-to-one correspondence between the primitive generators of $1$-cones in $\Sigma$ and the primitive generators of $1$-cones in $\Sigma'$, we have that $\overline{d\alpha}$ induces a one-to-one correspondence between the primitive generators of $1$-cones in $q(\Sigma)$ and the primitive generators of $1$-cones in $q_0(\Sigma')$. This implies that $\overline{d\alpha}$ is an isomorphism. Therefore, $(d\alpha)^{-1}(\h') = \h_0'$. Furthermore, since $d\alpha$ is surjective, its restriction $d\alpha|_{\h_0'}\colon\h_0'\to\h'$ is also surjective. For the corresponding Lie subgroups $H_0'=\exp_{G}{\h_0'}\subset G_0$ and $H'=\exp_{G}\h'\subset G$, the map $\alpha|_{H'_0}\colon H'_0\to H'$ is also surjective.
		
Let $L$ be a leaf of $F$. By definition, $L$ is an $H'$-orbit, that is, $L = H' \cdot x$ for some $x \in L$,
Since $f \colon M_0 \to M$ is a principal bundle, there exists $x_0 \in M_0$ such that $f(x_0) = x$. We need to show that $f^{-1}(L) = H_0'\cdot x_0$. The map $f$ is $\alpha$-equivariant and $\alpha (H_0') = H'$ by the previous paragraph, which implies that $H'_0\cdot x_0\subset f^{-1}(L)$. To show the opposite inclusion, let $x_0' \in f^{-1}(L)$. Then $f(x_0') \in L$. Hence, there exists $h' \in H'$ such that $x = h' \cdot f(x_0')$. Since $\alpha|_{H_0'}\colon H_0'\to H'$ is surjective, there exists ${h_0'} \in H_0'$ such that $\alpha({h_0'}) = h'$. Then we have $f({h_0'}\cdot x_0') = h' \cdot f(x_0') = x = f(x_0)$. Since $f$ is a principal $\Ker \alpha$-bundle, there exists $k \in \Ker \alpha$ such that $x_0 = k\cdot ({h_0'}\cdot x_0')$. Now $(d\alpha)^{-1}(\h') = \h_0'$ implies that $\Ker d\alpha\subset\h'_0$. Since $\Ker \alpha$ is connected, we obtain $\Ker\alpha \subset H_0'$. Therefore, $k\cdot {h_0'} \in H_0'$ and  $x_0'= (k {h_0'})^{-1} \cdot x_0 \in H_0' \cdot x_0$. Thus, $f^{-1}(L) =H_0'\cdot x_0$ is a leaf of~$F_0$. 
\end{proof}

\begin{thm}\label{thm:equivalence}
		Let $(M_1,G_1,y_1), (M_2,G_2,y_2) \in \mathscr{C}_1$ be complex manifolds with maximal torus actions. Let $F_1$ and $F_2$ be the canonical foliations on $M_1$ and $M_2$, respectively. If $(M_1,G_1,y_1)$ and $(M_2,G_2,y_2)$ are p-equivalent, then $(M_1,F_1)$ and $(M_2,F_2)$ are transversely equivalent. 
	\end{thm}
	\begin{proof}
		This follows from Lemma \ref{lemma:equivalence} immediately. 
	\end{proof}
    
The p-equivalence class of a maximal torus action is determined by the combinatorial data defined next.
    
\begin{defn}\label{defn:markedfan}
		A \emph{marked fan} is a quadruple $(\widetilde{V}, \widetilde{\Gamma}, \widetilde{\Sigma}, \widetilde{\lambda})$, where
		\begin{itemize}
			\item $\widetilde{V}$ is a finite dimensional $\R$-vector space; 
			\item $\widetilde{\Gamma}$ is a finitely generated subgroup of $\widetilde{V}$ that spans $\widetilde{V}$ linearly;
			\item $\widetilde{\Sigma}$ is a fan in $\widetilde V$ and each $1$-cone of $\widetilde\Sigma$ is generated by an element of~$\widetilde{\Gamma}$; 
			\item $\widetilde{\lambda}$ is a function $\widetilde{\lambda} \colon\widetilde{\Sigma}^{(1)} \to \widetilde{\Gamma}$, where $\widetilde{\Sigma}^{(1)}$ is the set of $1$-cones of~$\widetilde{\Sigma}$, and $\widetilde{\lambda} (\rho)$ is a generator of $\rho \in \widetilde{\Sigma}^{(1)}$. 
		\end{itemize}
We say that a marked fan $(\widetilde{V}, \widetilde{\Gamma}, \widetilde{\Sigma}, \widetilde{\lambda})$ is \emph{simplicial} (respectively, \emph{complete}) if the fan $\widetilde{\Sigma}$ is simplicial (respectively, complete). We denote by $\widetilde{\mathscr{C}}_2$ the class that consists of complete simplicial marked fans. 
		
		We say that marked fans $(\widetilde{V}_1,\widetilde{\Gamma}_1,\widetilde{\Sigma}_1,\widetilde{\lambda}_1)$ and $(\widetilde{V}_2,\widetilde{\Gamma}_2,\widetilde{\Sigma}_2,\widetilde{\lambda}_2)$ are \emph{isomorphic} if there exists a linear isomorphism $\varphi \colon \widetilde{V}_1\to \widetilde{V}_2$ such that 
		\begin{itemize}
			\item $\varphi (\widetilde{\Gamma}_1) = \widetilde{\Gamma}_2$;
			\item $\varphi$ induces an isomorphism of fans $\widetilde{\Sigma}_1$ and $\widetilde{\Sigma}_2$;
			\item $\widetilde{\lambda}_2 \circ \varphi|_{\widetilde{\Sigma}_1^{(1)}} = \varphi \circ \widetilde{\lambda}_1$. 
		\end{itemize}
\end{defn}
    
\begin{constr}[the marked fan data of a maximal torus action]  We define a map  
\[
  \widetilde{\mathscr{F}}_1 \colon \mathscr{C}_1 \to \widetilde{\mathscr{C}}_2
\]
that assigns to each complex manifold with maximal torus action $(M,G,y) \in \mathscr{C}_1$ a complete simplicial marked fan $(\widetilde{V}, \widetilde{\Gamma}, \widetilde{\Sigma}, \widetilde{\lambda}) \in \widetilde{\mathscr{C}}_2$ as follows. Set $(\Sigma,\h,G) = \mathscr{F}_1(M,G,y)$. As before, let $p \colon \g^\C \to \g$ be the projection, $\h' = p(\h)$ and $q\colon \g \to \g/\h'$ the quotient map. 
    For each $1$-cone $\rho \in \Sigma^{(1)}$, we denote by $\lambda(\rho) \in \Ker \exp_G$ the primitive generator of $\rho$. Now set $\widetilde{V} := \g/\h'$, $\widetilde{\Gamma} := q(\Ker \exp_G)$, $\widetilde{\Sigma} := q(\Sigma)$ and $\widetilde{\lambda}(q(\rho)) := q(\lambda(\rho))$ for $\rho \in \Sigma^{(1)}$. 
The properties of the map 
$\widetilde{\mathscr{F}}_1$ are described in the next two theorems.
\end{constr}    
    
\begin{thm}\label{thm:fundamental}
Let $(M_1,G_1,y_1),(M_2,G_2,y_2) \in \mathscr{C}_1$  be complex manifolds with maximal torus actions. Then, $(M_1,G_1,y_1)$ and $(M_2,G_2,y_2)$ are p-equivalent if and only if the marked fans $\widetilde{\mathscr{F}}_1(M_1,G_1,y_1)$ and $\widetilde{\mathscr{F}}_1(M_2,G_2,y_2)$ are isomorphic. 
\end{thm}

\begin{proof}
		For $j=0,1,2$, let $(\Sigma_j,\h_j,G_j) := \mathscr F_1 (M_j,G_j,y_j)$. Let $p_j \colon \g_j^\C \to \g_j$ be the projection, $\h_j':= p_j(\h_j)$, $q_j \colon \g_j \to \g_j/p_j(\h_j)$ the quotient map and $\exp_{G_j} \colon \g_j \to G_j$ the exponential map.
		
		Suppose that $(M_1,G_1,y_1)$ and $(M_2,G_2,y_2)$ are p-equivalent. Then there exists a triple $(M_0,G_0,y_0) \in \mathscr{C}_1$ and $(f_i,\alpha_i) \in \Hom _{\mathscr{C}_1}((M_0,G_0,y_0),(M_i,G_i,y_i))$ for $i=1,2$ such that $\Ker \alpha_i$ is connected and $f_i \colon M_0 \to M_i$ is a principal $\Ker \alpha_i$-bundle. The map $\overline{d\alpha_i} \colon \g_0/\h_0' \to \g_i/\h_i'$ is an isomorphism (see the proof of Lemma~\ref{lemma:equivalence}) and it induces an isomorphism between $\widetilde{\mathscr{F}}_1(M_0,G_0,y_0)$ and $\widetilde{\mathscr{F}}_1(M_i,G_i,y_i)$. Hence, $\widetilde{\mathscr{F}}_1(M_1,G_1,y_1)$ and $\widetilde{\mathscr{F}}_1(M_2,G_2,y_2)$ are isomorphic.

		Conversely, suppose that $\widetilde{\mathscr F}_1(M_1,G_1,y_1)$ and $\widetilde{\mathscr F}_1(M_2,G_2,y_2)$ are isomorphic. By definition, this means that there is a linear isomorphism $\varphi \colon \g_1/\h_1' \to \g_2/\h_2'$ such that
		\begin{itemize}
			\item $\varphi(q_1(\Ker \exp_{G_1})) = q_2(\Ker \exp_{G_2})$; 
			\item $\varphi$ induces an isomorphism of fans $q_1(\Sigma_1) \to q_2(\Sigma_2)$;
			\item $\widetilde{\lambda}_2 \circ \varphi|_{q_1(\Sigma_1^{(1)})} = \varphi \circ \widetilde{\lambda}_1$. 
		\end{itemize}
		We construct $(\Sigma_0, \h_0, G_0) \in \mathscr{C}_2$ and use \cite[Theorem 11.1]{ishida2013complex} to show that $(M_1,G_1,y_1)$ and $(M_2,G_2,y_2)$ are equivalent. 
		Define 
		\begin{equation*}
			\Gamma_0 := \{ (\gamma_1,\gamma_2) \in \Ker \exp_{G_1} \times \Ker \exp_{G_2} \mid \varphi(q_1(\gamma_1)) = q_2(\gamma_2)\}
		\end{equation*}
and denote by $\g_0$ the linear hull of the discrete subgroup $\Gamma_0$ in $\g_1 \times \g_2$.  Let $\exp_{G_1 \times G_2} \colon \g_1 \times \g_2 \to G_1 \times G_2$ be the exponential map. Then $G_0:= \exp_{G_1\times G_2} (\g_0)$ is a subtorus of $G_1 \times G_2$ and $\Gamma_0$ coincides with the kernel of $\exp_{G_0} \colon \g_0 \to G_0$. Since $q_1$ and $q_2$ are bijective on the sets of cones of the fans, for each $\sigma_1 \in \Sigma_1$ we can define $\varPhi(\sigma_1)=q_2^{-1}\circ \varphi \circ q_1(\sigma_1)\in\Sigma_2$.  This implies that, for each $\rho_1 \in \Sigma_1^{(1)}$, the element $\lambda_0(\rho_1) := (\lambda_1(\rho_1), \lambda_2(\varPhi (\rho_1)))$ is a primitive element of~$\Gamma_0$. 
		Suppose that a cone $\sigma_1\in \Sigma_1$ is spanned by $\rho_{1,1}, \dots, \rho_{1,k} \in \Sigma_1^{(1)}$. Then we denote by $\varPsi(\sigma_1)$ the cone in $\g_0$ spanned by $\lambda_0(\rho_{1,1}), \dots, \lambda_0(\rho_{1,k})$. Under this notation, we have a nonsingular fan $\Sigma_0 := \{ \varPsi({\sigma_1}) \subset \g_0 \mid \sigma_1 \in \Sigma_1\}$. Let $\alpha_i \colon G_0 \to G_i$ be the projection $G_1 \times G_2 \to G_i$ restricted to $G_0 \subset G_1 \times G_2$ for $i=1,2$. Then $\alpha_i \colon G_0 \to G_i$ is surjective and its differential $d\alpha_i \colon \mathfrak{g}_0 \to \mathfrak{g}_i$ induces a morphism of fans $\Sigma_0 \to \Sigma_i$ that is bijective on the sets of cones. In particular, $d\alpha_i \colon \mathfrak{g}_0 \to \mathfrak{g}_i$ induces a one-to-one correspondence between the primitive generators of $1$-cones in $\Sigma_0$ and the primitive generators of $1$-cones in $\Sigma_i$. 
		
Now define $\h_0 := (\h_1 \times \h_2) \cap \g_0^\C$. Then $\h_0$ is a $\C$-subspace of $\g_0^\C \subset \g_1^\C \times \g_2^\C$.  Moreover, the restriction $p_0|_{\h_0}$ of the projection $p_0 \colon \g_0^\C \to \g_0$ is injective because both $p_1|_{\h_1}$ and $p_2|_{\h_2}$ are injective. Put $\h_0' := p_0(\h_0)$. Let $q_0 \colon \g_0 \to \g_0/\h_0'$ be the quotient map. Since $q_i$ and $d\alpha_i$ both are surjective, the composite $q_i \circ d\alpha_i \colon \g_0 \to \g_i/\h_i'$ is surjective for $i=1,2$. We claim that
	\begin{equation}\label{eq:2equalities}
		\Ker q_1 \circ d\alpha_1 =  \Ker  q_2 \circ d\alpha_2 = \h_0'.
	\end{equation} 
The first equality above holds since $q_2\circ d\alpha_2 = \varphi \circ q_1 \circ d\alpha_1$ and $\varphi$ is an isomorphism. For the second equality of \eqref{eq:2equalities}, let $(\gamma_1,\gamma_2) \in \Ker q_2\circ d\alpha_2 $. Then we have $q_2(\gamma_2) = 0$ and hence $\gamma_2 \in \h_2'$. The identity $q_2\circ d\alpha_2 = \varphi  \circ q_1\circ  d\alpha_1$ implies $\varphi \circ q_1(\gamma_1) = 0$. Since $\varphi$ is an isomorphism, $q_1(\gamma_1)= 0$. Hence, $\gamma_1 \in \h_1'$. Therefore $(\gamma_1, \gamma_2) \in \h_0'$. We proved that $\Ker q_2  \circ d\alpha_2 \subset \h_0'$. For the opposite inclusion, let $(\gamma_1,\gamma_2) \in \h_0'$. Then $\gamma_1 \in \h_1'$ and $\gamma_2 \in \h_2'$. Then $q_2\circ d\alpha_2(\gamma_1,\gamma_2) = q_2(\gamma_2) = 0$, which implies $\Ker q_2\circ  d\alpha_2 \supset \h_0'$. 
	
	Since $d\alpha_1 \colon \g_0 \to \g_1$ is surjective and $\Ker q_1\circ d\alpha_1 = \h_0'$, we have that $d\alpha_1$ induces an isomorphism $\overline{d\alpha}_1 \colon \g_0/\h_0' \to \g_1/\h_1'$. Since the maps $\Sigma_0 \to \Sigma_1$ given by $\sigma_0 \mapsto d\alpha_1(\sigma_0)$ and $\Sigma_1 \to q_1(\Sigma_1)$ given by $\sigma_1 \mapsto q_1(\sigma_1)$ are bijective, the composite $\Sigma_0 \to q_1(\Sigma_1)$ given by $\sigma_0 \mapsto q_1 \circ d\alpha_1 (\sigma_0)$ is also bijective. Now, both $\overline{d\alpha}_1$ and $q_0 = (\overline{d\alpha}_1)^{-1}\circ q_1 \circ d\alpha_1$ are isomorphism, so that
\begin{equation*}
		q_0 (\Sigma_0) = \{ q_0(\sigma_0) \colon \sigma_0 \in \Sigma_0\}
\end{equation*}
is a complete fan in $\g_0/\h_0'$ and the map $\Sigma_0 \to q_0(\Sigma_0)$ given by $\sigma _0 \mapsto q_0(\sigma_0)$ is bijective. Therefore $(\Sigma_0,\h_0, G_0) \in \mathscr{C}_2$. 

Applying \cite[Theorem 11.1]{ishida2013complex} to the morphism $\alpha_i \colon (\Sigma_0,\h_0,G_0) \to (\Sigma_i,\h_i, G_i)$, we obtain a $\alpha_i$-equivariant principal $\Ker \alpha_i$-bundle $V_{\Sigma_0}/H_0 \to V_{\Sigma_i}/H_i \cong M_i$. It remains to show that $\Ker \alpha_i$ is connected for $i=1,2$. Since $\alpha_i \colon G_0 \to G_i$ is surjective, we have that $\Ker \alpha_i$ is connected if and only if $d\alpha_i(\Ker \exp_{G_0}) = \Ker \exp_{G_i}$. Recall that $\Ker \exp_{G_0} = \Gamma_0$. Take $\gamma_1 \in \Ker \exp_{G_1}$. Since $\varphi(q_1(\Ker \exp_{G_1})) = q_2(\Ker \exp_{G_2})$, there exists $\gamma_2 \in \Ker \exp_{G_2}$ such that $q_2(\gamma_2) = \varphi \circ q_1(\gamma_1)$. Then we have $(\gamma_1,\gamma_2) \in \Gamma_0$ with $d\alpha_1(\gamma_1,\gamma_2)=\gamma_1$, 
showing that $d\alpha_1 (\Ker \exp_{G_0}) = \Ker \exp_{G_1}$. Similarly, $d\alpha_2 (\Ker \exp_{G_0}) = \Ker \exp_{G_2}$. Thus, $\Ker \alpha_i$ is connected.
\end{proof}
	
The following result is an adaptation of~\cite[Theorem~11.2, Remark~11.3]{ishida2013complex} to our situation.

\begin{thm}\label{thm:ess.surjective}
The map $\widetilde{\mathscr{F}}_1 \colon \mathscr{C}_1 \to \widetilde{\mathscr{C}}_2$ is essentially surjective. Furthermore, for any marked fan $(\widetilde{V}, \widetilde{\Gamma}, \widetilde{\Sigma}, \widetilde{\lambda}) \in \widetilde{\mathscr{C}}_2$, there exists a moment-angle manifold $\momang$ with a $T^m$-invariant complex structure such that its marked fan $\widetilde{\mathscr{F}}_1(\momang, T^m, y)$ is isomorphic to $(\widetilde{V}, \widetilde{\Gamma}, \widetilde{\Sigma}, \widetilde{\lambda})$. 
\end{thm} 

\begin{proof}
Let $(\widetilde{V}, \widetilde{\Gamma}, \widetilde{\Sigma}, \widetilde{\lambda}) \in \widetilde{\mathscr{C}}_2$. Let $\widetilde\Sigma^{(1)}=
\{\widetilde{\rho}_1,\dots, \widetilde{\rho}_{m'}\}$ be the set of $1$-cones of~$\widetilde{\Sigma}$, and let $\SK$ be the underlying simplicial complex of $\widetilde\Sigma$, given by
\[			
  \SK := \{ \{i_1,\dots, i_k\} \subset \{1,\dots, m'\} \mid \rho_{i_1} + \dots    
  +\rho_{i_k} \in \widetilde{\Sigma}\}.
\]		
Put $\widetilde{\gamma}_j := \widetilde{\lambda}(\rho_j)$ for $j=1,\dots, m'$. Since $\widetilde{\Gamma}$ is finitely generated, we can choose elements $\widetilde{\gamma}_{m'+1}, \dots, \widetilde{\gamma}_{m}$, $m\geqslant m'$, such that $\widetilde{\gamma}_1,\dots, \widetilde{\gamma}_{m'}, \widetilde{\gamma}_{m'+1}, \dots, \widetilde{\gamma}_{m}$ generate $\widetilde{\Gamma}$ and $m-\dim\widetilde V$ is nonnegative and even. For $i=1,\dots, m$, let $\mb e_i$ denote the standard basis vectors of $\R^m$. The collection of cones $\Sigma _\SK := \{ \R_\ge\langle\mb e_i \colon i \in I\rangle \colon I \in \SK \}$ is the fan of the toric variety $U(\SK)$.  
		
Let $\Lambda \colon \R^m \to \widetilde{V}$ be the linear map given by $\Lambda (\mb e_i) = \widetilde{\gamma}_i$ for $i=1,\dots, m$. Then there exists a $\C$-subspace $\h$ of $\C^m$ such that $\mathop{\mathrm{Re}} \colon \C^m \to \R^m$ restricted to $\h$ is injective and $\mathop{\mathrm{Re}}(\h) = \Ker \Lambda$ (see Construction~\ref{czk}). Therefore $(\Sigma, \h, T^m) \in \mathscr{C}_1$, and the moment-angle manifold $\momang = U(\SK)/H$ has the required properties.
\end{proof}

\begin{rmk}
Theorem~\ref{thm:ess.surjective} also implies that the map $\widetilde{\mathscr{F}}_1 \colon \mathscr{C}_1 \to \widetilde{\mathscr{C}}_2$ restricted to LVMB-manifolds is essentially surjective, because every complex moment-angle manifold is an LVMB-manifold.
\end{rmk}

We can finally describe the basic cohomology ring of an arbitrary complex manifold with a maximal torus action:

\begin{thm}\label{basiccohgen}
Let $(M,G,y) \in \mathscr{C}_1$ be a complex manifold with a maximal torus action, and let $(\widetilde{V}, \widetilde{\Gamma}, \widetilde{\Sigma}, \widetilde{\lambda}) = \widetilde{\mathscr{F}}_1(M,G,y)$ be the corresponding marked fan data. Let $\widetilde\Sigma^{(1)}=
\{\widetilde{\rho}_1,\dots, \widetilde{\rho}_{m}\}$ be the set of $1$-cones of~$\widetilde{\Sigma}$. There is an isomorphism of algebras:
		\[
		  \hb^*(M) \cong \mathbb{R}[v_1,\ldots,v_m] /(I_{\SK}+J),
		\]
		where $I_\SK$ is the Stanley--Reisner ideal of the underlying simplicial complex of $\widetilde{\Sigma}$, and $J$ is the ideal generated by the linear forms
		\[
		  \sum_{i=1}^m \langle \mb u, \widetilde\lambda(\widetilde\rho_i) \rangle v_i,\quad \mb u\in \widetilde{V}^*.
		\]
\end{thm} 
\begin{proof}
By Theorem~\ref{thm:ess.surjective}, for the manifold $M$, there exist a moment-angle manifold~$\momang$ with isomorphic marked fan data. By Theorem~\ref{thm:fundamental}, the manifolds $\momang$ and $M$ are p-equivalent as manifolds with maximal torus actions. By Proposition~\ref{bcohinv}, their basic cohomology algebras are isomorphic. Finally, the basic cohomology algebra of~$\momang$ is described by Theorem~\ref{basicTh}.
\end{proof}

\end{document}